\documentclass[11pt]{amsart}

\usepackage{standard}
\usepackage{tikz}
\newcommand{\tV}{\Upsilon} 
\newcommand{\hexpt}{\nu}
\newcommand{\wexpt}{\lambda}

\usepackage{xcolor}

\thanks{The authors are grateful to Luc Hillairet for comments on the
  manuscript and for helpful conversations; in particular, JW wishes
  to acknowledge the strong influence of Hillairet's point of view on
  this subject on his own, formed in the course of writing
  \cite{HiWu:17}. Thanks also to Ethan Brady and the anonymous referee
for helpful comments and corrections. This collaboration came out of a discussion had
  through a virtual meeting hosted by the Casa Matem\'atica  Oaxaca;
  the authors are grateful to the institution for hosting them.  K.D.\
  acknowledges support from NSF grant DMS-1708511.
  J.L.M.\ acknowledges support from NSF grant DMS-1909035. J.W.\ was
  partially supported by Simons Foundation grant 631302, NSF grant
  DMS--2054424, and a Simons Fellowship.  }

\author[K. Datchev]{Kiril Datchev} \address{Mathematics
  Department, Purdue University, West Lafayette,
  IN 47907, USA} \email{kdatchev@purdue.edu}

\author[J.~L.~Marzuola]{Jeremy~L.~Marzuola} \address{Mathematics
  Department, University of North Carolina at Chapel Hill, Chapel Hill,
  NC 27599, USA} \email{marzuola@math.unc.edu}
  
  \author[J. Wunsch]{Jared Wunsch} \address{Mathematics
  Department, Northwestern University, Evanston,
  IL 60208, USA} \email{jwunsch@math.northwestern.edu}

\usepackage[letterpaper,hmargin=1.5in,vmargin=1.2in]{geometry} 

\title{Newton polygons and resonances of multiple delta-potentials}
\date{\today}

\begin{document}

\begin{abstract}
We prove explicit asymptotics for the location of semiclassical
  scattering resonances in the
  setting of $h$-dependent delta-function potentials on
    $\mathbb{R}$.  In the cases of two or three delta poles, we are
  able to show that resonances occur along specific lines of the form
  $\Im z \sim -\gamma h \log(1/h).$  More generally, we use the method
  of Newton polygons to show that resonances near the real axis may
  only occur along a finite collection of such lines, and we bound the
  possible number of values of the parameter $\gamma.$  We present numerical
  evidence of the existence of more and more possible values of
  $\gamma$ for larger numbers of delta poles.
\end{abstract}

\maketitle

\section{Introduction}

We consider certain ``leaky'' semiclassical quantum systems where most of the energy
escapes to infinity but some $h$-dependent fraction is trapped.  In
such settings, it has often been observed that strings of resonances
occur along curves $\Im z \sim -\gamma h \log (1/h)$ for certain values
of $\gamma$ related to the geometry.  This has been observed, with
varying degrees of precision, 
in scattering with nonsmooth potentials on the real line
\cite{Regge:Analytic}, \cite{Zworski4}; scattering by multiple delta
singularities in $\RR^3$ \cite{Ze:01};
  scattering between a corner and an
analytic obstacle \cite{Burq:Coin}; scattering on a manifold with
conic singularities \cite{dkk}, \cite{Ga:15}, \cite{HiWu:17}; and scattering by
thin barriers, modeled by $h$-dependent $\delta$-potentials
\cite{DaMa:22}, \cite{Ga:14}.  In some of these settings where the
geometry of trapping is relatively simple, e.g., 
 \cite{Burq:Coin}, \cite{DaMa:22}, the structure of all resonances near the real axis can
be precisely understood, with one or more strings of resonances
occurring at \begin{equation}\label{Im}\Im z \sim -\gamma h \log
  (1/h),\end{equation} for certain values of $\gamma$ and no others present.  More generally, however, the picture is
muddier, with some information known about $O(h \log (1/h))$-width
resonance-free regions near $\RR$ and in some cases about existence of
a limited region in which the resonances are distributed as in
\eqref{Im}.

Here we analyze a situation in which the geometry of trapping is
complicated enough to generate multiple strings of resonances of the
form \eqref{Im}, and moreover for that structure to vary interestingly
as we tune the parameters of the problem.  This is the situation of
several thin barriers on $\RR,$ modeled by potentials of the form
$h^{1+\beta}\delta(x)$, $\beta > 0.$ 
One dimensional problems with delta function barriers have been studied before in \cite[Section II.2]{aghh}, \cite{DaMa:22}, \cite{duchene2011wave}, \cite{herbstmavi}, \cite{sacchetti}, \cite{Tanimu_2018}  but only the second reference considered our asymptotic regime, and that only in a very special case. In the case of two and three
delta poles, we are
able to analyze the distribution of resonances very precisely: in the
former case, there is a single curve of resonances near the real axis
with $\Im z \sim -\gamma h \log (1/h)$ (Theorem~\ref{t:2delta}); in the
latter, there may be either one or two such families instead
(Theorem~\ref{t:3deltagen}). In particular, in the latter case there is one family 
if the deltas all have equal strength. In \cite[Appendix A]{Ze:01}, Zerzeri computes
resonances of multiple delta poles in $\mathbb R^3$ and finds analogously that
they  are all asymptotic to a single logarithmic curve.

In the more general case of $N$
$\delta$-potentials,  we are able to constrain the locations of
resonances by analyzing the secular determinant that governs their
locations in terms of its Newton polygon.  We show
(Theorem~\ref{t:Ndelta}) that in any set $\Im z \geq -M h \log(1/h)$
there may be no more than $2^{N-1}-1$
possible values for the parameter $\gamma$ in \eqref{Im}, and that all
such possible values may be simply expressed in terms of the various strengths
$\beta$ of the potential poles and differences of distances among them.

\section{General setup}
Consider the semiclassical Hamiltonian on the real line
  $$
P=-h^2\pa_x^2 + V(x), \qquad V(x)=\sum_{j=1}^{N} V_j\delta(x-x_j), \qquad h>0,
  $$
where $x_1<\cdots<x_N$, and each $V_j = C_j h^{1+\beta_j}$ for some $C_j \in \mathbb R \setminus\{0\}$ and $\beta_j >0$.

A \textit{resonant state} $u$ is an outgoing distributional solution to
\begin{equation}\label{eigenfunction}
(-h^2 \pa_x^2 + V -z^2)u=0,
\end{equation}
and a \textit{resonance} is a value of $z \in \mathbb C$ for which a resonant state exists. More explicitly, define $I_j$ for $j=0,\dots, N$,    by $I_0=(-\infty, x_1]$,  $I_j=[x_j, x_{j+1}]$ when $1\le j\le N-1$, and $I_{N}=[x_N,
    +\infty)$. If \eqref{eigenfunction} holds in the sense of distributions, then $u= v_j^+ e^{izx/h} + v_j^- e^{-izx/h}$ on $I_j$, with appropriate continuity and jump conditions (which we state in \eqref{e:continuity} and \eqref{e:jump} below) at each $x_j$. Such a solution $u$ is \textit{outgoing} if it is not identically zero and if $v_N^-= v_0^+=0$. See Section 2.1 et seq. of \cite{dz19} for an introduction to resonances.

For \eqref{eigenfunction} to hold we need $u$ to be continuous at each $x_j$, i.e., the continuity condition is
\begin{equation}\label{e:continuity}
-v_{j-1}^- e^{-i x_j z/h}- v_{j-1}^+ e^{+i x_j z/h}+  v_j^- e^{-i x_j z/h}+ v_j^+ e^{+i x_j z/h}=0.
\end{equation}
Moreover, $u'$ must have a jump at each $x_j$ so that $(h^2\pa_x^2+z^2) u$ contains a multiple of $\delta(x-x_j)$ which equals $V_ju(x_j)\delta(x-x_j)$. That leads to the jump condition
\begin{equation}\label{e:jump}\begin{split}
& \frac{hz}{i} \left(v_{j-1}^-e^{-izx_j/h}-v_{j-1}^+e^{+izx_j/h}
  -v_j^- e^{-izx_j/h}+ v_j^+e^{+izx_j/h}\right)  \\
  & \hspace{4.5cm} + V_j  ( v_j^- e^{-i x_j z/h}+ v_j^+ e^{+i x_j z/h})=0.
\end{split}\end{equation}

To bring the continuity and jump conditions \eqref{e:continuity} and \eqref{e:jump} to a more manageable form, we now require $z \ne 0$, set 
  \begin{equation}\label{Vtilde} \tV_j = \frac {V_j}{2izh} = \frac{C_j h^{\beta_j}}{2iz},\end{equation}
  and take
  \begin{equation*}
w =e^{-iz/h}, \quad
  y_j^+ =v_j^+ e^{i x_{j} z/h}, \quad
    y_j^- =v_j^- e^{-i x_{j+1} z/h}.
\end{equation*}
  These are the values of the amplitudes immediately following interaction with the potential poles.

Let $\ell_j= x_{j+1}-x_j=\abs{I_j}.$  
  Our continuity and jump equations \eqref{e:continuity} and \eqref{e:jump}  now read
  $$
\left\{   \begin{aligned}
 &   - y_{j-1}^- - y_{j-1}^+ e^{i \ell_{j-1} z/h} + y_j^- e^{i \ell_j z/h} + y_j^+ =0, \\
& y_{j-1}^- - y_{j-1}^+ e^{i \ell_{j-1} z/h} + y_j^-e^{i \ell_j z/h}(-1-2 \tV_j)  + y_j^+(1- 2 \tV_j) =0.
\end{aligned} \right.
$$
    Adding these equations yields
\begin{equation}\label{eq1}
y_j^+ = T_j e^{i \ell_{j-1}z/h} y_{j-1}^+ + R_j e^{i \ell_j z/h} y_j^-
\end{equation}
    with
    \begin{equation}\label{TR}
T_j = \frac{1}{1-\tV_j},\ R_j=\frac{\tV_j}{1-\tV_j}.
\end{equation}
Subtracting $1-2 \tV_j$ times the first from the second yields likewise
\begin{equation}\label{eq2}
y_{j-1}^-= T_je^{i \ell_j z/h} y_j^-+R_j e^{i\ell_{j-1} z/h}  y_{j-1}^+.
\end{equation}
In the extreme cases $j=0$ or $N$ we simply get the special cases where there is no reflection:
$$
y_N^+=T_N e^{i \ell_{N-1}z/h} y_{N-1}^+ 
$$
and
$$
y_0^-=T_1e^{i \ell_1 z/h} y_1^-.
$$
Note that these components are completely determined by the others.

\section{Logarithmic strings for two and three deltas}

In this section we  consider the simpler cases $N=2$ and $N=3$, in which our description of the resonances is more complete.

\subsection{Two deltas}
Let  $N=2$, and put $\ell_1 = \ell$.

\begin{theorem}\label{t:2delta}
 All resonances obeying $1/2 \le |z| \le 2$ and $\operatorname{Re} z > 0$ are given by
\begin{equation}\label{e:2deltaexp}
 z_k = \frac{\pi h k}{\ell} - i \frac {\beta_1+\beta_2}{2\ell} h \log (1/h) + O(h),
\end{equation}
for some positive integers $k$. Moreover, for any $\delta$ such that  $\delta<1$ and $\delta \le \min(\beta_1,\beta_2)$ we have
\begin{equation}\label{e:2deltarefre}
\operatorname{Re} z_k = \frac {\pi h} \ell  \Big(k + \frac{H(C_1C_2)}2  + O(h^\delta)\Big), 
\end{equation}
where $H$ is the Heaviside function, and
\begin{equation}\label{e:2deltarefim}
 \operatorname{Im} z_k = \frac h {2\ell} \Bigg(-(\beta_1+\beta_2) \log (1/h) +  \log \Big( \frac{|C_1C_2| \ell^2}{4\pi^2h^2k^2}  \Big) + O(h^\delta) \Bigg).
\end{equation}

\end{theorem}

\begin{proof}
In this case we have
\[
 y_1^+ = R_1 e^{i\ell z/h} y_1^-, \qquad y_1^- = R_2 e^{i\ell z/h} y_1^+,
\]
and so resonances occur if and only if
\begin{equation}\label{e:2dreseq}
 e^{-2i\ell z/h} = R_1R_2.
\end{equation}
Take the logarithm of both sides of \eqref{e:2dreseq} and multiply through by $ih/2\ell$ to get
\begin{equation}\label{e:2dlogeq}
 z = \frac {ih}{2\ell} \log (R_1R_2) + \frac{\pi h k}{\ell},
\end{equation}
where $k$ is an integer. Substituting
\begin{equation}\label{e:logr1r2}
 \log (R_1R_2) = \log \Big( \frac{-C_1C_2 h^{\beta_1 + \beta_2}}{4z^2(1-\tV_1)(1-\tV_2)}  \Big) = -(\beta_1+\beta_2)  \log(1/h) + O(1),
\end{equation}
into \eqref{e:2dlogeq} gives
\[
 z = \frac{\pi h k}{\ell} - i \frac {\beta_1+\beta_2}{2\ell} h \log (1/h) + O(h).
\]
It is clear that if $\pi h k \le \ell/3$ or $\pi h k \ge 3 \ell$ then the right hand side is not in $\{z \in \mathbb C \colon \Re z > 0 \text{ and } 1/2 \le |z| \le 2\}$ for $h$ small.  Hence, to establish \eqref{e:2deltaexp}, is enough to prove that, for $h$ small enough, if $k$ is such that  $\ell/3 \le \pi h k \le 3\ell$, then \eqref{e:2dlogeq} has a unique solution $z$ in the half-annulus $A = \{z \in \mathbb C \colon \Re z \ge 0 \text{ and } 1/4 \le |z| \le 4\}$.

For this we apply Rouch\'e's theorem (the Corollary of Section 5.2 of \cite{Ahlfors}) with $f(z) = z - \tfrac 1 \ell \pi h k$ and $g(z) = \frac {ih}{2\ell} \log (R_1R_2) $  on the half-annulus $A$ (note that $g$ is analytic on $A$ by Corollary 2 of Section 4.4 of \cite{Ahlfors}). Since $f(z) = 0$ obviously has a unique solution in $A$, it is enough to check that $|g(z)| < |f(z)|$ on $\partial A$.  For that, note that on $\partial A$ we have $|f(z)| \ge 1/12$ and use \eqref{e:logr1r2}.

Finally, to get \eqref{e:2deltarefre} and \eqref{e:2deltarefim}, note
that \eqref{e:2dreseq}, \eqref{e:logr1r2} yield
$z^2 = \frac{\pi^2 h^2 k^2}{\ell^2} + O(h \log(1/h));$ since
$\tV_j = O(h^{\beta_j}),$
\begin{align*}
& \log(R_1R_2) = \log \Big(\frac{-C_1C_2 h^{\beta_1 + \beta_2}}{4z^2(1-\tV_1)(1-\tV_2)}  \Big)  \\
 & \hspace{2cm} = -(\beta_1+\beta_2)  \log(1/h) + \log \Big( \frac{-C_1C_2 \ell^2}{4\pi^2h^2k^2}  \Big) + O(h^\delta).
\end{align*}
Inserting this into \eqref{e:2dlogeq} yields \eqref{e:2deltarefre}, \eqref{e:2deltarefim}.
\end{proof}

\subsection{Three deltas}

For $N=3$ we use $w=e^{-iz/h}$ and write
\[
 y_2^- = R_3 w^{-\ell_2} y_2^+, \qquad y_1^+ = R_1 w^{-\ell_1} y_1^-,
\]
and plugging those into the equations for $y_2^+$ and $y_1^-$ gives
\[\begin{split}
 y_2^+ &= R_1 T_2 w^{-2\ell_1} y_1^- + R_2 R_3w^{-2\ell_2} y_2^+, \\
y_1^- &= R_1R_2 w^{-2\ell_1}y_1^-  + T_2R_3 w^{-2\ell_2}y_2^+.
\end{split}\]

These equations have a nontrivial solution if and only if
\[R_1R_2^2R_3 w^{-2\ell_1-2\ell_2} - R_1R_2 w^{-2\ell_1} - R_1T_2^2R_3 w^{-2\ell_1-2\ell_2} - R_2R_3 w^{-2\ell_2}+1 = 0.
\]
Since $T_2 = R_2 + 1$, another way to write this is
\begin{equation}\label{e:3dreseq}
 w^{2\ell_1+2\ell_2} - R_1R_2 w^{2\ell_2} - R_2 R_3 w^{2\ell_1} - R_1(1+2R_2)R_3 = 0.
\end{equation}

If the delta functions are equally spaced, this can be solved using the quadratic formula and works out similarly to the case of two deltas.

\begin{theorem}\label{t:3delta=}
 If $\ell_1=\ell_2 =\ell$, then there are positive real numbers
 $\gamma_+$ and $\gamma_-$ (which may or may not be distinct,
 depending on $\ell, \ \beta_1, \ \beta_2, \ \beta_3$), such that all
 resonances obeying $1/2 \le |z| \le 2$ and $\operatorname{Re} z > 0$
 are given by
\[
 z^+_k = \frac{\pi h k}{\ell} - i \gamma_+ h\log (1/h) + O(h), \qquad  z^-_k = \frac{\pi h k}{\ell} - i \gamma_- h\log (1/h) + O(h),
\]
for some positive integers $k$.
\end{theorem}

Simple explicit formulas for the $\gamma_\pm$ can be obtained either by elaborating the calulation in the proof of Theorem \ref{t:3delta=} (which is a more complicated version of the one in Theorem \ref{t:2delta}), or as special cases of the ones in Theorem \ref{t:3deltagen} below.  More precise asymptotics for the real and imaginary parts of $z_k^\pm$, as in \eqref{e:2deltarefre} and \eqref{e:2deltarefim}, follow as in the proof of Theorem \ref{t:2delta}.

\begin{proof}
By the quadratic formula, \eqref{e:3dreseq} is equivalent to 
\[
( w^{2\ell} - r_-)(w^{2\ell} - r_+) = 0,
\]
where
\[
 r_{\pm} = \frac 12 \Big( (R_1+R_3)R_2 \pm \sqrt{ (R_1+R_3)^2R_2^2 + 4 R_1(1+2R_2)R_3}\ \Big).
\]
By the same argument as in the proof of Theorem \ref{t:2delta} we get strings of resonances
\[
 z^\pm_k = \frac{\pi h k}{\ell} + \frac {ih}{2\ell} \log r_\pm,
\]
where $\log r_\pm = - \gamma_{\pm} \ell \log(1/h) + O(1)$ for some $\gamma_\pm > 0$. 

There are various ways to choose the $\beta_j$ so as to make either  $\gamma_+ \ne \gamma_-$ or  $\gamma_+ = \gamma_-$. For example, if $\beta_1 + 2 \beta_2 < \beta_3$, then $R_3 = O(R_1R_2^2h^\delta)$ for some $\delta>0$, and thus $r_\pm = \frac 12(R_1 R_2 \pm R_1 R_2 + O(R_1R_2h^\delta)),$ and  $\gamma_+ \ne \gamma_-$. 

On the other hand, if for example $\beta_1 < \beta_3 < \beta_1 + 2 \beta_2$, then we get $r_{\pm} \sim \pm
\sqrt{R_1R_3}$ and hence $\gamma_+ = \gamma_-$. 
\end{proof}

Our next theorem gives necessary conditions on the logarithmic
curves the resonances can approach when $\ell_1$ is not necessarily equal to $\ell_2$.

\begin{theorem}\label{t:3deltagen}
  Let $\ell_1, \ \ell_2, \ \beta_1, \ \beta_2, \ \beta_3$ be
  given. Let $h_1, \ h_2, \ \dots$ be a sequence of positive numbers
  tending to $0$. Let $z = z(h_j)$ be a sequence of resonances such
  that $z = h^{o(1)}$ (i.e. such that $z(h_j) = e^{f(h_j)}$ for some $f \colon (0,h_1] \to \mathbb C$ obeying $|f(h)| = o(\log (1/h))$) and $\Im z \ge - M h \log(1/h) $ for some
  positive $M$. Then this sequence has a subsequence such that
\begin{equation}\label{e:3dimzgamma}
 \frac{\Im z}{h \log(1/h)} \to - \gamma, 
\end{equation}
for some $\gamma \in \{\gamma_+, \gamma_-\}$, where $\gamma_+$ and $\gamma_-$ are determined as follows:

\begin{enumerate}
 \item If $\beta_3\ell_1 - \beta_2\ell_1 - \beta_2\ell_2 \le  \beta_1 \ell_2 \le  \beta_2\ell_1 + \beta_2\ell_2 + \beta_3\ell_1$, then
\[
 \gamma_+ = \gamma_- = \frac{\beta_1+\beta_3}{2\ell_1+2\ell_2}.
\]
\item If 
$\beta_3\ell_1 - \beta_2\ell_1 - \beta_2\ell_2 >  \beta_1 \ell_2$,
 then
 \[
  \gamma_+ = \frac{\beta_3 - \beta_2}{2\ell_2} > \gamma_- = \frac{\beta_1 + \beta_2}{2\ell_1}.
 \]
\item If $\beta_2\ell_1 + \beta_2\ell_2 + \beta_3\ell_1 < \beta_1 \ell_2$, then
\[
 \gamma_+ = \frac{\beta_1 - \beta_2}{2\ell_1} > \gamma_- = \frac{\beta_2 + \beta_3}{2\ell_2}.
\]
\end{enumerate}
\end{theorem}

\begin{remark}\label{rem:simpler}
  Note that because the resonances of $ - h^2 \partial_x^2 +  V(-x)$ are the same as the resonances of  $ - h^2 \partial_x^2 +  V(x)$, it is no loss of generality to make the simplifying assumption
\begin{equation}\label{e:b3l1b1l2}
\beta_3\ell_1 \le \beta_1\ell_2.
\end{equation}
Then the three cases of the theorem reduce to the following two:
\begin{enumerate}
 \item If $ \beta_1 \ell_2  \le  \beta_2\ell_1 + \beta_2\ell_2 + \beta_3\ell_1$, then
\[
 \gamma_+ = \gamma_- = \frac{\beta_1+\beta_3}{2\ell_1+2\ell_2}.
\]
\item If $\beta_2\ell_1 + \beta_2\ell_2 + \beta_3\ell_1 < \beta_1 \ell_2$, then
\[
 \gamma_+ = \frac{\beta_1 - \beta_2}{2\ell_1} > \gamma_- = \frac{\beta_2 + \beta_3}{2\ell_2}.
\]
\end{enumerate}
One can interpret (2) as corresponding to the case in which the middle delta is strong enough to split the interval $(x_1,x_3)$ at $x_2$, and  (1) as corresponding to the case in which it is not.
\end{remark}

Note also that in each of the limiting situations $\beta_1 \to \infty$
or $\beta_2 \to \infty$, which each correspond to one of the delta
functions becoming vanishingly small, the resonances converge to those
of a two-delta problem as in Theorem \ref{t:2delta}.

\begin{proof}
As noted above, we may without loss of generality proceed under
  the assumption \eqref{e:b3l1b1l2}, and show the simpler version of
  the theorem in Remark~\ref{rem:simpler}.

After passing to a subsequence, we have $\Im z /h \log(1/h) \to - \gamma$ for some $\gamma \in [-\infty,M]$, and so $w = h^{\gamma+o(1)}$.  By the reflection coefficient formulas \eqref{Vtilde} and \eqref{TR} and using $z = h^{o(1)}$, we have $\tV_j = \frac {C_j}{2i} h^{\beta_j + o(1)}$ and hence $ R_j = h^{\beta_j + o(1)},$
and thus the resonance equation \eqref{e:3dreseq} implies that $ \gamma > 0$. 

Next, we eliminate $y_1^-$ and $y_2^+$ from \eqref{eq1} and \eqref{eq2}, the equations for the $y_j^\pm$, by writing
\[
 y_2^- = R_3 w^{-\ell_2} y_2^+, \qquad y_1^+ = R_1 w^{-\ell_1} y_1^-,
\]
which gives
\[\begin{split}
 y_2^+ &= R_1 T_2 w^{-2\ell_1} y_1^- + R_2 R_3w^{-2\ell_2} y_2^+, \\
y_1^- &= R_1R_2 w^{-2\ell_1}y_1^-  + T_2R_3 w^{-2\ell_2}y_2^+.
\end{split}\]
We now substitute $w =   h^{\gamma + o(1)}$, $T_j = h^{o(1)}$, and $R_j =   h^{\beta_j + o(1)}$.

That gives
\begin{align}
\label{e:y2+=} y_2^+ &=   h^{\beta_1-2\ell_1\gamma + o(1)} y_1^- +  h^{\beta_2 + \beta_3-2\ell_2\gamma + o(1)}  y_2^+, \\
\label{e:y1-=}y_1^- &=  h^{\beta_1 + \beta_2 -2\ell_1\gamma + o(1)}y_1^-  + h^{\beta_3-2\ell_2\gamma + o(1)}y_2^+.
\end{align}
We now consider three cases according to whether the $y_2^+$ terms on the left and on the right of \eqref{e:y2+=} have comparable sizes or whether one dominates the other.

Case I. If the sizes are comparable,
i.e., if \begin{equation}\label{cond1}\beta_2 + \beta_3 - 2 \ell_2
  \gamma = 0\end{equation} then we use \eqref{e:3dreseq}, and observe
that the $w^{2\ell_1 + 2\ell_2}$ and $R_2R_3 w^{2\ell_1}$ terms both equal $h^{ \beta_2+\beta_3 + 2 \ell_1 \gamma+ o(1)}$, 
$R_1R_2 w^{2\ell_2}=h^{\beta_1+2\beta_2 + \beta_3+ o(1)}$, and  $R_1(1+2R_2)R_3=h^{\beta_1+\beta_3+ o(1)}$. So the $w^{2\ell_1 + 2\ell_2}$
and $R_2R_3 w^{2\ell_1}$ terms need to be at least as big as the
$R_1(1+2R_2)R_3$ term, and they need to cancel one another; the former
condition means we need
$\beta_2 + \beta_3 + 2\ell_1 \gamma \le \beta_1 + \beta_3$, i.e., by \eqref{cond1},
$\beta_2 \ell_2 + \beta_2\ell_1 + \beta_3\ell_1 \le \beta_1\ell_2$.

Case II. If the term on the right is dominant, i.e., if $\beta_2+\beta_3 - 2\ell_2\gamma<0$, then \eqref{e:y2+=} becomes
\[
 y_2^+ =  h^{\beta_1 - \beta_2 - \beta_3 + 2 \ell_2 \gamma - 2 \ell_1 \gamma + o(1)}y_1^-,
\]
which, inserted into \eqref{e:y1-=}, gives
\[
y_1^- =  h^{\beta_1 + \beta_2 -2\ell_1\gamma + o(1)}y_1^-  +  h^{\beta_1 - \beta_2-2\ell_1\gamma + o(1)}y_1^-  = h^{\beta_1 - \beta_2-2\ell_1\gamma + o(1)}y_1^-.
\]
Hence  $0=\beta_1-\beta_2 - 2 \ell_1 \gamma$, or $\gamma = \frac{\beta_1-\beta_2}{2\ell_1}$. This requires $(\beta_2+\beta_3)\ell_1 < (\beta_1-\beta_2)\ell_2$.

Case III: If the term on the left is dominant, i.e., if  $\beta_2+\beta_3 - 2\ell_2\gamma>0$, then \eqref{e:y2+=} becomes
\[
 y_2^+ =  h^{\beta_1-2\ell_1\gamma + o(1)} y_1^-,
\]
which, inserted into \eqref{e:y1-=}, gives
\[
 y_1^- =   h^{\beta_1 + \beta_2 -2\ell_1\gamma + o(1)}y_1^-  + h^{\beta_1 + \beta_3-2\ell_1 \gamma - 2\ell_2\gamma + o(1)}y_1^- .
\]
Of these three terms, two must be of the same size and the other must be no bigger. We accordingly have three subcases.

Subcase 1: If the term on the left is the small one, then $\beta_1+\beta_2 - 2 \ell_1 \gamma = \beta_1 + \beta_3 - 2 \ell_1 \gamma - 2 \ell _2 \gamma \le 0$. That means $\gamma = \frac{\beta_3 - \beta_2}{2\ell_2}$ and we require $\beta_1 \ell_2 + \beta_2 \ell_2 + \beta_2 \ell_1 \le \beta_3 \ell_1$. This contradicts \eqref{e:b3l1b1l2}.

Subcase 2: If the first term on the right is the small one, then $0 = \beta_1 + \beta_3 - 2 \ell_1 \gamma - 2 \ell _2 \gamma \le \beta_1 + \beta _2 - 2 \ell_1 \gamma$. That means $\gamma = \frac {\beta_1+\beta_3}{2\ell_1+2\ell_2}$, and then we need $(\beta_2 + \beta_3)(\ell_1+\ell_2) > (\beta_1+\beta_3) \ell_2  \ge (\beta_3 - \beta_2)(\ell_1+\ell_2)$ which is equivalent to $\beta_2\ell_1 + \beta_2\ell_2 + \beta_3\ell_1 > \beta_1 \ell_2  \ge \beta_3\ell_1 - \beta_2\ell_1 - \beta_2\ell_2$.

Subcase 3: If the second term on the right is the small one, then $0 = \beta_1 + \beta _2 - 2 \ell_1 \gamma \le \beta_1 + \beta_3 - 2 \ell_1 \gamma - 2 \ell _2 \gamma$. That means $\gamma = \frac{\beta_1+\beta_2}{2\ell_1}$ and we  require $\ell_2 (\beta_1+\beta_2) \le (\beta_3 - \beta_2)\ell_1$. This contradicts \eqref{e:b3l1b1l2}.

In summary, under the assumption \eqref{e:b3l1b1l2}, we have three possible values of $\gamma$, each with a corresponding necessary condition on the coefficients:
\begin{itemize}
\item If  $\gamma = \frac {\beta_2 + \beta_3}{2 \ell_2}$, then $\beta_2 \ell_1 + \beta_2\ell_2 + \beta_3\ell_1 \le \beta_1\ell_2$.
\item If $\gamma = \frac{\beta_1 - \beta_2}{2\ell_1}$, then $\beta_2 \ell_1 + \beta_2 \ell_2 + \beta_3 \ell_1 < \beta_1 \ell_2$.
\item If $\gamma = \frac {\beta_1+\beta_3}{2\ell_1+2\ell_2}$, then $\beta_1 \ell_2  <  \beta_2\ell_1 + \beta_2\ell_2 + \beta_3\ell_1$.
\end{itemize}
The conclusions of the theorem follow from these.
\end{proof}


\section{$N$ deltas}

In this section we generalize the observations in the special cases of
two and three delta-poles to the general case of $N$ poles: the main
tool, as before, is simply examination of the leading terms in the
(generally transcendental) equations determining their location.  To
aid in understanding those terms, we begin by introducing the
machinery of Newton polygons, a traditional tool in the study of
resolution of plane
algebraic curves which also applies in the setting studied here.

\subsection{Newton polygons}

Here we explore how Newton polygons apply to the analysis of equations
of a form generalizing e.g.\ \eqref{e:3dreseq} which (as we will show
below) arise in the study of the more general case.
In particular, let
  $$p(h,w) =\sum_{j=0}^N  h^{{\hexpt}_j + o(1)} w^{\wexpt_j},$$
  where all exponents ${\hexpt}_j,$ $\wexpt_j$ are real and
  nonnegative.  Note, we take the polygon to contain the semi-infinite horizontal
    and vertical segments, which seems to differ slightly from standard conventions in the Newton Polygon literature.
    
  \begin{definition}
The Newton polygon is of $p$ is the boundary of the convex hull of the union of the first quadrants displaced to have vertices at the points $({\hexpt}_j, \wexpt_j)$: 
$$
\pa \conv \bigcup_{j=0}^N (({\hexpt}_j, \wexpt_j)+ [0,\infty)^2).
$$
\end{definition}
See Figure \ref{fig:newton} for an example, and see \cite[Section 8.3]{BrKn:86} for the classical algebro-geometric
theory of Newton polygons.

\begin{lemma}\label{lemma:newton}
Consider the equation
\begin{equation}\label{e:phw}
p(h,w) \equiv \sum_{j=0}^N h^{{\hexpt}_j + o(1)} w^{\wexpt_j}=0,
\end{equation}
  where all exponents ${\hexpt}_j,$ $\wexpt_j$ are real and
  nonnegative. Suppose that, for all $j \ge 1$, we have $\hexpt_j > \hexpt_0$ and $\wexpt_j < \wexpt_0$.
Fix any $M>0.$ Then any sequence of
  roots of the equation $w=w(h)$ for $h \in (0,1)$ with
  $\abs{w}^{-1}=O(h^{-M})$ has a
  subsequence that asymptotically satisfies (as $h \downarrow 0$)
  $$
\log \abs{w} \sim {\gamma} \log h
$$
where $-1/\gamma$ is one of the finitely  many nonzero slopes occurring in the Newton polygon of $p.$    \end{lemma}

\begin{proof}
  Say we have a family of solutions $w=w(h)$ with $h \downarrow 0$
  and (without loss of generality) with $\abs{w}^{-1}\leq h^{-M}.$  Then since $\log \abs{w}\geq M \log h,$ the ratio $\log \abs{w}/\log h$ lies in $(-\infty, M)$, hence along a subsequence, $\log\abs{w}/\log h$ converges to $\gamma \in [-\infty, M].$  

We first rule out the case $\gamma \le 0$, much as in the proof of
Theorem~\ref{t:3deltagen}. If $\gamma \le 0$ then
for any $\ep>0,$ $\abs{w}>h^\ep$ for $h$ sufficiently small.  Hence
for all $j \geq 1,$ choosing $\ep$ sufficiently small yields
$$
\frac{h^{{\hexpt}_j + o(1)} \abs{w}^{\wexpt_j}}{h^{{\hexpt}_0 + o(1)}
  \abs{w}^{\wexpt_0}} \leq h^{\hexpt_j-\hexpt_0+o(1)}
h^{-\ep(\wexpt_0-\wexpt_j)} \to 0.
$$
Thus
the term in $p$
given by $h^{{\hexpt}_0+o(1)} w^{\wexpt_0}$ is dominant, and it cannot
be cancelled by the other terms and hence \eqref{e:phw} cannot hold.
Hence we may take $\gamma>0$ finite and assume, passing to our
subsequence, that $\log \abs w=(\gamma+o(1))\log h.$

Then
\begin{equation}\label{bigsum}
\sum_{j=0}^N  h^{{\hexpt}_j+\gamma \wexpt_j + o(1)} = 0. 
\end{equation}
If the minimum exponent ${\hexpt}_j+\gamma \wexpt_j$ occurring in the sum is unique, then as $h \downarrow 0,$ the term $h^{{\hexpt}_j+\gamma \wexpt_j + o(1)}$ dominates all other terms in the sum for $h$ sufficiently small, 
hence \eqref{e:phw} again cannot hold.
So 
the minimum exponent in \eqref{bigsum} must occur in at least two terms, say $j$ and $k;$ in particular, then,
$$
{\hexpt}_j+\gamma \wexpt_j={\hexpt}_k +\gamma \wexpt_k,
$$
and $\gamma=-({\hexpt}_k-{\hexpt}_j)/(\wexpt_k-\wexpt_j)$ is the negative reciprocal of the slope of the line connecting these two points.

We claim that the minimality of the exponent $$\rho \equiv
{\hexpt}_j+\gamma \wexpt_j={\hexpt}_k+\gamma \wexpt_k$$ further
entails that the segment $\overline{({\hexpt}_j,
  \wexpt_j)({\hexpt}_k,\wexpt_k)}$ is in the Newton
polygon, which will complete our characterization of $\gamma$ as the
negative reciprocal of the slope of a segment of the 
Newton polygon.  To see this, we first observe that minimality of
$\rho$ means for every $i,$ $\rho \leq {\hexpt}_i+\gamma \wexpt_i.$
Since for every $s \in \RR,$ 
$$
\rho=s {\hexpt}_j + (1-s) {\hexpt}_k + \gamma (s \wexpt_j + (1-s) \wexpt_k)
$$
the point $s ({\hexpt}_j,\wexpt_j) + (1-s) ({\hexpt}_k, \wexpt_k)$
cannot lie in the quadrant $({\hexpt}_i, \wexpt_i)+ (0, \infty)^2,$ as
this would imply $\rho> {\hexpt}_i+\gamma \wexpt_i.$  Thus we have
shown that minimality of $\rho$ means that the interior of every
quadrant $({\hexpt}_i, \wexpt_i)+ [0,\infty)^2$ lies above the
line $$L\equiv \{s ({\hexpt}_j,\wexpt_j) + (1-s) ({\hexpt}_k,
\wexpt_k) \setcolon s \in \RR\},$$ hence the convex hull of the
quadrants $({\hexpt}_i, \wexpt_i)+ [0,\infty)^2$ lies entirely in the
closed half-space above $L.$  Since the segment $\overline{({\hexpt}_j,
  \wexpt_j)({\hexpt}_k,\wexpt_k)}$ of $L$ does lie in the convex hull of
the vertices, it must be in the Newton polygon, as asserted.
\end{proof}

\subsection{Analysis of the secular determinant}

We now employ the method of Newton polygons introduced above to
analyze the case of $N$ delta poles; the main problem is to find a good
description of the secular determinant arising in the equations for a
putative resonant state.

In the following, we employ multiindex notation for combinations of exponents $\beta_j$ ($j=1,\dots, N$) and lengths $\ell_j$ ($j =1, \dots, N$), e.g.\ writing $\sigma \cdot \beta=\sum_j \sigma_j \beta_j.$

Note that our result on this general case of $N$ deltas, like our  Theorem \ref{t:3deltagen} on three arbitrarily spaced deltas, focuses on
resonances in a
narrower region of $\CC$ than in Theorems \ref{t:2delta} and \ref{t:3delta=}: the
imaginary part is a priori $O(h\log(1/h))$.
\begin{theorem}\label{t:Ndelta}
Consider the Hamiltonian on the real line
  $$
P=-h^2\pa_x^2 + V(x), \qquad V(x)=\sum_{j=1}^{N} V_j\delta(x-x_j),
  $$
where $x_1<\cdots<x_N$, and each $V_j = C_j h^{1+\beta_j}$ for some $C_j \in \mathbb R$ and $\beta_j >0$.  

Let $z = z(h_j)$ be a sequence of resonances such
  that $z = h^{o(1)}$ (as in Theorem \ref{t:3deltagen}) and $\Im z \ge - M h \log(1/h) $ for some
  positive $M$. Then this sequence has a subsequence such that
$\Im z \sim -\gamma h \log (1/h)$ where $\gamma$ is one
  of at most $2^{N-1}-1$ values.  All possible values of $\gamma$ are
positive numbers of the form
$$
\frac{\sigma^+ \cdot \beta -\sigma^-\cdot \beta}{2 (\alpha^+ \cdot \ell -\alpha^- \cdot \ell)}
$$
for some $\sigma^\pm \in \{0, 1, 2\}^N$ and $\alpha^\pm \in \{0, 1\}^{N-1}$, where $\ell_j=x_{j+1}-x_j.$

  \end{theorem}
  
  \begin{remark}
As of the publication of this paper, Theorem~\ref{t:Ndelta} has been
considerably refined in work of Brady \cite{brady2023resonances}, who obtains sharp
estimates on the number of values of $\gamma$ that may arise.
  \end{remark}
  
  \begin{proof}
   Setting $w=e^{-iz/h}$, we recall that the condition $\Im z
      \ge - M h \log(1/h) $ yields $\abs{w}^{-1} \leq h^{-M}.$

      We collect the components $( y_1^-, y_1^+,\dots, y_{N-1}^-, y_{N-1}^+)$ into a vector, which lies in the nullspace of $A_N-I$ where $A_N$ is the $2(N-1)\times 2(N-1)$ matrix given by the equations \eqref{eq1}, \eqref{eq2}:
$$\tiny
\begin{pmatrix}
0 & R_2 w^{-\ell_1} & T_2 w^{-\ell_2} & 0 & 0 & \hdots & 0 &0 &0\\
R_1 w^{-\ell_1} & 0& 0 & 0 & 0 & \hdots & 0 &0 &0\\
0& 0& 0 & R_3 w^{-\ell_2} & T_3 w^{-\ell_3} & \hdots & 0 &0 &0\\
0& T_2 w^{-\ell_1} & R_2 w^{-\ell_2} & 0 & 0 & \hdots & 0 &0 &0\\
\vdots & \vdots  & \vdots & \vdots & \vdots & \ddots & \vdots &\vdots &\vdots\\
0& 0& 0& 0 & 0 & \hdots & 0 &0 &R_N w^{-\ell_{N-1}}\\
0& 0& 0& 0 & 0 & \hdots & T_{N-1} w^{-\ell_{N-2}}&R_{N-1} w^{-\ell_{N-1}} &0
  \end{pmatrix}.
$$
The general pattern is that of a pentadiagonal matrix with zeros on the diagonal and overlapping blocks
$$
\begin{pmatrix}
0 & 0 & R_{j+1} w^{-\ell_{j}} & T_{j+1} w^{-\ell_{j+1}}\\
T_j w^{-\ell_{j-1}} & R_j w^{-\ell_j} & 0 & 0
  \end{pmatrix},
$$
which arises in the rows $y_j^-, y_j^+$ and columns $y_{j-1}^+, y_j^-, y_j^+, y_{j+1}^-.$

Note that in the base case $N=2$ we get the matrix
$$
\begin{pmatrix}
0 & R_2 w^{-\ell_1} \\
R_1 w^{-\ell_1} & 0
\end{pmatrix}
$$
and \begin{equation}\label{Neq2}\det(A_2-I)=1-R_1R_2 w^{-2\ell_1}.\end{equation}
We claim that just as in this example, we always get only even powers of $w^{-\ell_j},$ i.e.\ that we may, more generally, express
\begin{equation}\label{formofdet}
\det(A_N-I)=\sum_{\alpha \in \{0, 1\}^{N-1}} a_\alpha w^{-2 \alpha \cdot \ell}
\end{equation}
where the coefficients $a_\alpha$ are composed of (unspecified) sums of products of $T_j$ and $R_j$'s. This will follow from the following more general lemma.
\begin{lemma}
\label{lem:secdet}
  Let $W_N$ be a $2(N-1)\times 2(N-1)$ matrix of the form
  $$\tiny
\begin{pmatrix}
\sigma_1 & R_2 w^{-\ell_1} & T_2 w^{-\ell_2} & 0 & 0 & \hdots & 0 &0 &0\\
R_1 w^{-\ell_1} & \sigma_2 & 0 & 0 & 0 & \hdots & 0 &0 &0\\
0& 0& \sigma_3 & R_3 w^{-\ell_2} & T_3 w^{-\ell_3} & \hdots & 0 &0 &0\\
0& T_2 w^{-\ell_1} & R_2 w^{-\ell_2} & \sigma_4 & 0 & \hdots & 0 &0 &0\\
\vdots & \vdots  & \vdots & \vdots & \vdots & \ddots & \vdots &\vdots &\vdots\\
0& 0& 0& 0 & 0 & \hdots & 0 &\sigma_{2N-3} &R_N w^{-\ell_{N-1}}\\
0& 0& 0& 0 & 0 & \hdots & T_{N-1} w^{-\ell_{N-2}}&R_{N-1} w^{-\ell_{N-1}}&\sigma_{2N-2}
  \end{pmatrix}
$$
where each $\sigma_j \in \{0, -1\}$.  Then
$\det W_N$ is of the form.
\begin{equation}\label{formofdet2}
\sum_{\alpha \in \{0, 1\}^{N-1}} a_\alpha w^{-2 \alpha \cdot \ell}
\end{equation}
where each $a_\alpha$ is a sum of products of $T_j$ and $R_j$'s.
\end{lemma}
  The greater generality of taking $\sigma_j$ terms on the diagonal rather than all $-1$'s is of no interest except that it enables the following inductive proof to work.
  \begin{proof}[Proof of Lemma]
    The result holds for $N=2$ since we get
    $\sigma_1 \sigma_2- R_1 R_2 w^{-2\ell_1}.$ We now proceed
    inductively. For brevity we denote an entry of the form
      $R_i w^{-\ell_j}$ or $T_i w^{-\ell_j}$
    simply $L_j$ (as we will never employ any cancellation among
    terms, the ambiguity in the index $i$ and the difference
      between $T_i$ and $R_i$ are of no importance); we
    also write $\pm$ to be independent and completely unimportant
    signs in the following computation.  We simply need to show that
    each $L_j$ appears in each summand in the determinant either not
    at all or as $L_j^2.$

In our abbreviated notation, we now have
$$
W_N=\begin{pmatrix} 
\sigma_1 & L_1 & L_2 & 0 & 0 & \hdots & 0 &0 & 0\\
L_1 & \sigma_2 & 0 & 0 & 0 & \hdots & 0 &0 &0\\
0 & 0& \sigma_3 & L_2 & L_3 & \hdots & 0 &0 &0\\
0&L_1 & L_2 & \sigma_4 & 0 & \hdots & 0 &0 &0\\
\vdots & \vdots & \vdots  & \vdots & \vdots & \vdots & \ddots & \vdots &\vdots \\
0 & 0& 0& 0 & 0 & \hdots & 0 &\sigma_{2N-3} &L_{N-1}\\
0& 0& 0& 0 & 0 & \hdots & L_{N-2} & L_{N-1} &\sigma_{2N-2}
\end{pmatrix}.
$$    

    Decomposing $W_N$ by cofactors in the first column yields
\begin{equation}\begin{aligned}
\sigma_1 &\det \begin{pmatrix} \sigma_2 & 0 & 0 & 0 & \hdots & 0 &0 &0\\
 0& \sigma_3 & L_2 & L_3 & \hdots & 0 &0 &0\\
L_1 & L_2 & \sigma_4 & 0 & \hdots & 0 &0 &0\\
\vdots & \vdots  & \vdots & \vdots & \vdots & \ddots & \vdots &\vdots \\
 0& 0& 0 & 0 & \hdots & 0 &\sigma_{2N-3} &L_{N-1}\\
 0& 0& 0 & 0 & \hdots & L_{N-2} & L_{N-1} &\sigma_{2N-2}
\end{pmatrix}\\
& - L_1
\det \begin{pmatrix} L_1 & L_2 & 0 & 0 & \hdots & 0 &0 &0\\
 0& \sigma_3 & L_2 & L_3 & \hdots & 0 &0 &0\\
L_1 & L_2 & \sigma_4 & 0 & \hdots & 0 &0 &0\\
\vdots & \vdots  & \vdots & \vdots & \vdots & \ddots & \vdots &\vdots\\
 0& 0& 0 & 0 & \hdots & 0 &\sigma_{2N-3} &L_{N-1}\\
 0& 0& 0 & 0 & \hdots & L_{N-2} & L_{N-1} &\sigma_{2N-2}
\end{pmatrix}\\
& \hspace{.25cm} \equiv \sigma_1 \det B_N - L_1 \det C_N.
\end{aligned}\end{equation}
We deal with these terms as follows.  Decomposing $B_N$ further by cofactors in its first row gives a single term that equals $\sigma_1 \sigma_2$ times the determinant of a matrix of the form $W_{N-1},$ which by the inductive hypothesis is a sum of terms of the form coefficient times $L_2^{2 \alpha_2}\dots L_{N-1}^{2 \alpha_{N-1}}$ with $\alpha_j \in \{0, 1\};$ hence this term is of the desired form.

  Likewise, decomposing $\det C_N$ by cofactors in the first column gives, from the top left $L_1$ entry, a term $L_1^2$ times a term of the form $\det W_{N-1},$ hence yields a sum of terms $L_1^2 L_2^{2 \alpha_2}\dots L_{N-1}^{2 \alpha_{N-1}}$ by the inductive hypothesis.  Finally the $L_1$ entry in position $(3,1)$ gives a term
    $$
L_1^2
\det \begin{pmatrix}  L_2 & 0 & 0 & \hdots & 0 &0 &0\\
 \sigma_3 & L_2 & L_3 & \hdots & 0 &0 &0\\
 0& 0 & \sigma_5 & L_3 & L_4 & \hdots & 0 \\
 0& L_2 & L_3 & \sigma_6 & 0 & \hdots & 0 \\
 \vdots  & \vdots & \vdots & \vdots & \ddots & \vdots &\vdots\\
 0& 0 & 0 & \hdots & 0 &\sigma_{2N-3} &L_{N-1}\\
 0& 0 & 0 & \hdots & L_{N-2} & L_{N-1} &\sigma_{2N-2}
\end{pmatrix}
\equiv L_1^2 \det D_N  .   $$
Now exchanging the first two rows of $D_N$ gives another matrix of the form $W_{N-1}$ but where the $(2,2)$ entry is necessarily $0$ (rather than allowed to be $-1$).  Thus by induction, this term is also of the desired form.  Note that it was this last case that necessitated allowing the more general $\sigma_j$ entries on the diagonal in the inductive hypothesis.  This completes the proof of the lemma.
\end{proof}

We have now established \eqref{formofdet}.  Recall that the
coefficients $a_\alpha$ are of the form of $\pm$ products of
reflection and transmission coefficients $T_j$ and $R_j$ given by
\eqref{TR}, with $\tV_j$ given by \eqref{Vtilde}.  In the region
  $z=h^{o(1)}$, we have $\tV_j = h^{\beta_j+o(1)}$, hence 
  \begin{equation}\label{TR2}\begin{aligned} T_j&=1+h^{\beta_j + o(1)},\\
 R_j&= h^{\beta_j +
  o(1)},\end{aligned}\end{equation}
which implies that the terms $a_\alpha$ are all of the form
$h^{\mu + o(1)}$ for some values of $\mu$ given by sums of
powers $\beta_j$ occurring in the reflection coefficients $R_j.$ Since
each $R_j$ appears in at most two rows, we note that the only
possibilities for the appearance of $R_j$ in a coefficient $a_\alpha$
are as $R_j^{\sigma_j}$ with $\sigma_j\in\{0, 1, 2\}.$

We have now established that the equation
$$
\det(A_N-I)=0
$$
is of the form
\begin{equation}\label{formofeq}
\sum_{\alpha \in \{0, 1\}^{N-1}}  h^{\mu_\alpha + o(1)} w^{-2\alpha\cdot \ell}=0.
\end{equation}
We now claim further that all terms except the term $1$
are of the form $h^{\mu_\alpha + o(1)} w^{-2\alpha\cdot \ell}$
where $\alpha \neq 0$ and $\mu_\alpha>0.$  By \eqref{TR2}, this
follows from the following lemma.  As above we use the notation $L_i$
to be either $T_i w^{-\ell_j}$ for some $i,j,$ or $R_i w^{-\ell_j}.$
\begin{lemma}
Every term in the secular determinant $\det (A_N-I)$ except the
diagonal term $1$
is of the form $R_j E w^{-\alpha\cdot \ell}$ where $E$ is some
product of reflection and transmission coefficients and $\alpha \neq 0.$
  \end{lemma}
In other words, each nonconstant term has at least one reflection
coefficient and a negative power of $w.$
\begin{proof}
We again work by induction.  By \eqref{Neq2}, the result certainly
holds for $N=2.$  Cofactor decomposition as above in the first column
then yields
$$
\det (A_N-I)=(-1) \det B_N - R_1 w^{-\ell_1} \det C_N.
$$
Since $\det C_N$ is a sum of product of reflection and transmission
coefficients and negative powers of $w$, the second term certainly
satisfies the desired conclusion, so we need only examine the first.
The matrix $B_N$ is given by
$$\tiny
B_N = \begin{pmatrix} -1 & 0 & 0 & 0 & \hdots & 0 &0 &0\\
 0& -1 & R_3 w^{-\ell_2} & T_3 w^{-\ell_3} & \hdots & 0 &0 &0\\
T_2 w^{-\ell_1} & R_2w^{-\ell_2} & -1 & 0 & \hdots & 0 &0 &0\\
\vdots & \vdots  & \vdots & \vdots & \vdots & \ddots & \vdots &\vdots \\
 0& 0& 0 & 0 & \hdots & 0 &-1 &R_Nw^{-\ell_{N-1}}\\
 0& 0& 0 & 0 & \hdots & T_{N-1} w^{-\ell_{N-2}}&R_{N-1} w^{-\ell_{N-1}} &-1
\end{pmatrix}.
$$
Cofactor expansion in the first row now allows us to write $\det
B_N=-\det(B_{N-1})$, a secular determinant of the
same form as $\det(A_N-I),$ hence the lemma now follows by induction.
\end{proof}

Now we return to the representation \eqref{formofeq} of the secular determinant.
Multiplying through by $w^{2\abs{\ell}}$ (with $\abs\ell\equiv \sum \ell_j$ in multiindex notation) gives an equation with positive powers of $w:$ 
\begin{equation}\label{finalmess}
\sum_{\alpha \in \{0, 1\}^{N-1}} h^{\mu_\alpha + o(1)} w^{2(\abs\ell-\alpha\cdot \ell)}=0.
\end{equation}
Here the leading powers $\mu_\alpha$ are all sums of powers
arising in the delta potentials of the form
$$\mu=\sigma \cdot \beta$$ for $\sigma \in \{0, 1, 2\}^N.$
Moreover by the preceding lemma there is  a ``leading'' term $h^0 w^{2\abs{\ell}},$ with all
other terms having \emph{both} a higher power of $h$ and a lower power
of $w.$

Thus
Lemma~\ref{lemma:newton} applies to show that any sequence of
solutions to this equation with $\abs{w}^{-1}=O(h^{-M})$ has a subsequence with
$\log{\abs{w}} \sim \gamma \log h$ for $\gamma$ a strictly negative reciprocal
slope arising in the Newton polygon associated to the function
\eqref{finalmess}.  Since there are at most $2^{N-1}$ distinct powers
of $w$ in this equation there are at most $2^{N-1}-1$ nonzero finite
slopes in the Newton polygon, and $\gamma$ may only take the negative
reciprocal of one of these values.  Note further that owing to our
characterization of the exponents of $h$ and $w$, all $\gamma$'s are
thus of the form
$$
\frac{\sigma^+ \cdot \beta -\sigma^-\cdot \beta}{2 (\alpha^+ \cdot \ell -\alpha^- \cdot \ell)}
$$
for some $\sigma^\pm \in \{0, 1, 2\}^N$ and $\alpha^\pm \in \{0, 1\}^{N-1}.$

Now given
$$
\log \abs{w} \sim\gamma \log h
$$
and $w=e^{-iz/h}$ we of course get
$$
\Im z \sim - \gamma h \log (1/h)
$$
as desired.
\end{proof}

\begin{remark}
It is instructive to compare the general result of
Theorem~\ref{t:Ndelta} to the special cases of two and three poles
analyzed above.  In the case of two delta poles,
Theorem~\ref{t:Ndelta} correctly implies that as $h\downarrow 0$ there can be at most a
single curve of resonances $\Im z\sim -\gamma h \log (1/h)$ within any
set $\Im z >-M h \log (1/h)$ with $M$ fixed.  In the
case of three deltas, however, the bound given by this theorem is that
there can be at most $3$ such curves, while Theorem~\ref{t:3deltagen}
shows that $2$ is in fact the sharp maximum number of resonance lines.  This discrepancy is
clearer if we examine the Newton polygon for \eqref{e:3dreseq}:
recalling that $R_j\sim C_j h^{\beta_j}$ we see that the vertices
involved are
$$
(\beta_1+\beta_3, 0),\ (\beta_2+\beta_3, 2\ell_1),\ (\beta_1+\beta_2,
2\ell_2),\ (0, 2\ell_1+2\ell_2).
$$
A priori, this many vertices could yield a Newton diagram with $3$
nonvanishing finite slopes, hence we conclude naively from
Theorem~\ref{t:Ndelta} that there could be at most $3$ possible
values of $\gamma$.  Note, encouragingly, that the form of the secular
determinant established in the proof of Theorem~\ref{t:Ndelta} is indeed
giving the sharp overall form of the equation \eqref{e:3dreseq}.  
But
it turns out on closer inspection of the equation  that not every
possible Newton polygon can arise here.  In particular, under the
assumption \eqref{e:b3l1b1l2} (which we recall is always valid up to
reversing the $x$-coordinate), the vertex $(\beta_1+\beta_2, 2
\ell_2)$ always lies strictly above the line
$\overline{(0,2\ell_1+2\ell_2,0)(\beta_1+\beta_3, 0)}$, hence cannot
lie in the Newton polygon.  Thus there can be either two or one
nonzero finite slopes
in the Newton polygon, depending on whether $ (\beta_2+\beta_3,
2\ell_1)$ lies below or above this line; this is determined by the
condition
$$
 \beta_1 \ell_2  \gtrless  \beta_2\ell_1 + \beta_2\ell_2 + \beta_3\ell_1,
$$
i.e.\ agrees with the analysis of the cases in
Remark~\ref{rem:simpler}.  (See Figure~\ref{fig:newton}.)
\end{remark}

\begin{figure}[h!]\begin{center}
    \begin{tikzpicture}[scale=.5]
      \def\diam{0.1}
      \def\os{2.5}
      \draw[thick,->] (0, 0) -- (0, 10);
  \draw[thick,->] (0, 0) -- (10,0);
  \filldraw (0, 8) circle (\diam);
  \node[align=left] at (0+\os, 8){$(0, 2\ell_1+2\ell_2)$};
  \filldraw (7,0) circle (\diam);
    \node[align=left] at (7+\os/2, 0+\os/3){$(\beta_1+\beta_3,0)$};
    \filldraw (5,5) circle (\diam);
      \node[align=left] at (5+\os, 5){$(\beta_1+\beta_2, 2\ell_2)$};
  \filldraw (2,3) circle (\diam);
  \node[align=left] at (2+\os, 3){$(\beta_2+\beta_3, 2\ell_1)$};
  \draw[blue, very thick, dashed] (0, 8) -- (2,3);
  \draw[blue, very thick, dashed] (2,3) -- (7,0);
  \draw[blue, very thick, dashed] (7,0) -- (10,0);
    \draw[blue, very thick, dashed] (0,8) -- (0,10);
\filldraw[fill=red, opacity=0.3] (0,10) -- (0,8) -- (2,3) -- (7,0) -- (10,0) -- (10,10) -- (0,10);
\end{tikzpicture}
\end{center}
\caption{\label{fig:newton}
Newton polygon for the case $N=3$.  The Newton polygon is the union of
the dashed lines forming the boundary of the shaded region.  The point $(\beta_1+\beta_2,
2\ell_2)$ does not lie on the boundary of the shaded region, i.e., is
not in the Newton polygon.  This depicts the case $\beta_2\ell_1 +
\beta_2\ell_2 + \beta_3\ell_1 < \beta_1 \ell_2$, which guarantees that
$(\beta_2+\beta_3, 2\ell_1)$ does lie in the Newton polygon, hence two
distinct nonzero finite slopes arise.
}
\end{figure}
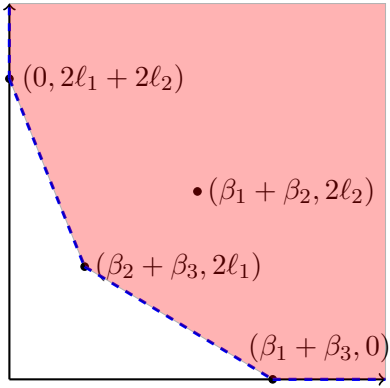

\section{Some numerical studies and discussion of the results}

We can program the secular determinant matrix $A_N = A_N(z)$ from the proof of Theorem~\ref{t:Ndelta} into the software program {\it Mathematica}, and study the resulting complex equations
\begin{equation}\label{e:detia}
 \det (I - A_N(z)) = 0.
\end{equation}
Resonances occur at solutions to \eqref{e:detia}.  It is particularly informative to plot the argument of the left hand side of \eqref{e:detia}; then poles become clear points about which the phase angle winds.  Such plots allow us to numerically observe the results in Theorems \ref{t:2delta}, \ref{t:3delta=}, and \ref{t:3deltagen} in the case of $2$ or $3$ delta functions, and to test the bounds of what we can prove in the general case in Theorem \ref{t:Ndelta}.  Our findings are presented in Figure \ref{fig:23delta} and Figure \ref{fig:Ndelta} respectively.  Throughout, we have taken $h = 10^{-6}$ and plotted the argument on a region of the complex plane such that $1-3h < \Re z < 1+3h$ and $-3h < \Im z < 0$.  

\begin{figure}[ht]
\centering
\includegraphics[width=0.45\textwidth]{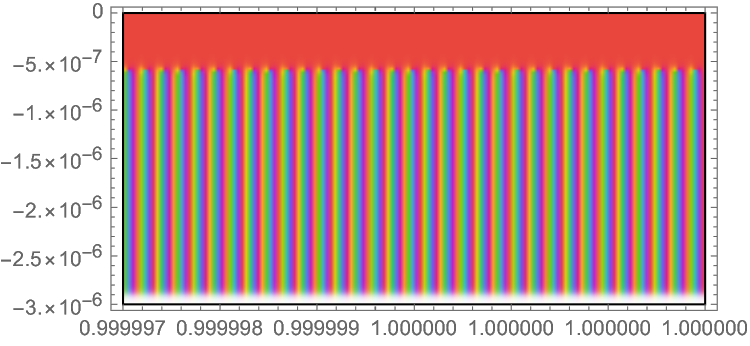}
\includegraphics[width=0.10\textwidth]{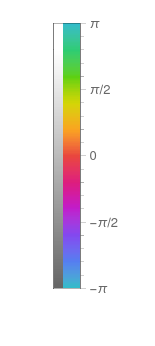} \\
\includegraphics[width=0.45\textwidth]{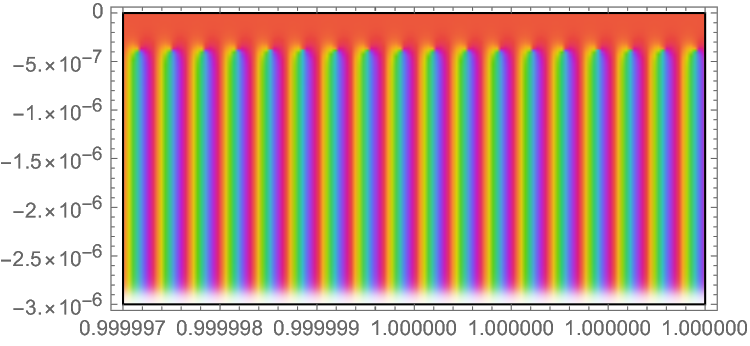} 
\includegraphics[width=0.45\textwidth]{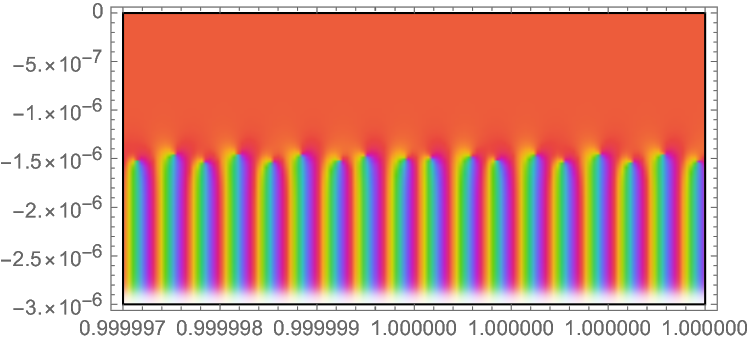} 
\caption{(Top)  A plot showing the resonances arising in the setting of $N=2$ delta functions, and a legend for the plot showing the correspondence between colors and complex arguments of the left hand side of \eqref{e:detia}.  (Bottom) The cases of $N=3$ delta functions in the setting of  one line of resonances from Theorem \ref{t:3deltagen}, Case $(1)$ (Left) and two lines of resonances from Theorem \ref{t:3deltagen}, Case $(2)$ (Right).  }
\label{fig:23delta}
\end{figure}

\begin{figure}[ht]
\centering
\includegraphics[width=0.45\textwidth]{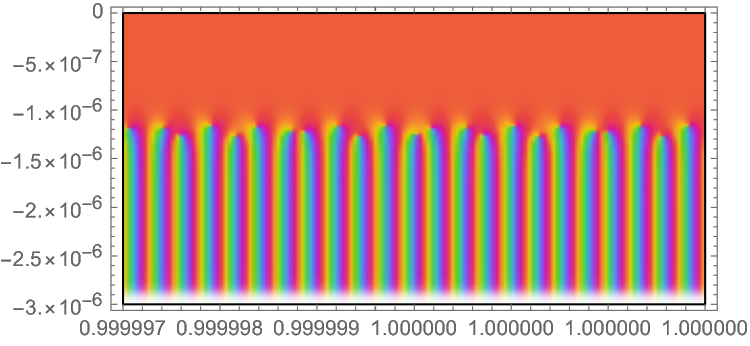} 
\includegraphics[width=0.45\textwidth]{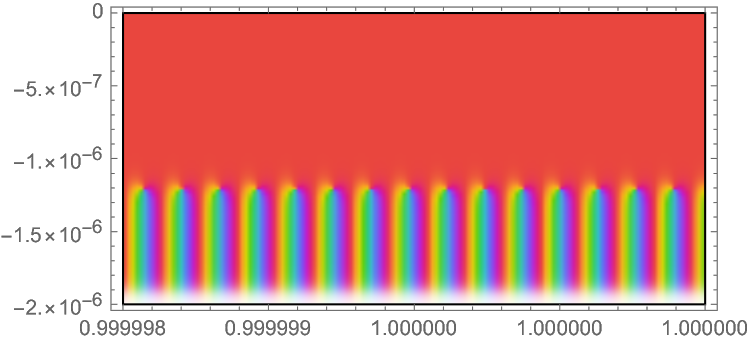} \\
\includegraphics[width=0.45\textwidth]{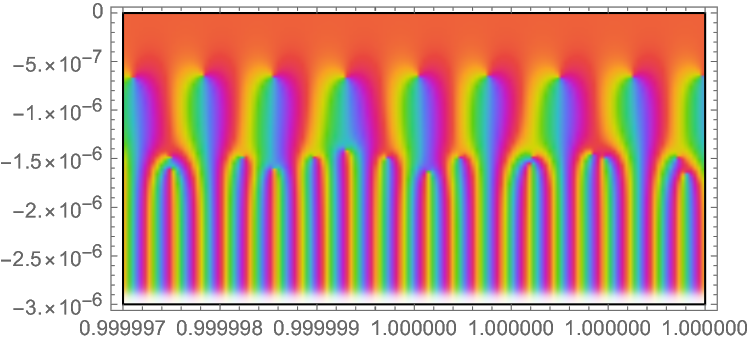} 
\includegraphics[width=0.45\textwidth]{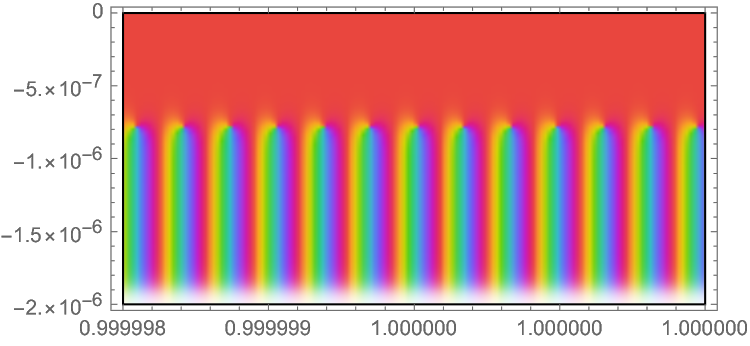} 
\caption{(Top)  A plot showing the resonances arising in the setting of $N=5$.  (Bottom) The cases of $N=6$.  In both cases, by varying values of $\beta$ and $\ell$, we can generate either multiple resonance lines or only one line.}
\label{fig:Ndelta}
\end{figure}

In Figure \ref{fig:23delta}, we observe that for $N=2$, we have one line of resonances as in Theorem \ref{t:2delta}.  This line is demonstrated in the top image with $\beta_1 = 1$, $\beta_2 = .5$, $x_1 = -10$ and $x_2 = 5 \sqrt{2}$, which give
\[
\Im z \sim   -\frac {\beta_1 +\beta_2}{2(x_2-x_1)} h \log(1/h)  \approx -6 \cdot 10^{-7}.
\]
  Meanwhile for $N=3$, we can choose $\beta_1, \beta_2, \beta_3$ such that we are either in the setting of Theorem \ref{t:3deltagen}, Case $(1)$ (bottom left) or Theorem \ref{t:3deltagen}, Case $(2)$ (bottom right).  In these cases, we took $x_1 = -5$, $x_2 = 0$, $x_3 = 3 \sqrt{2}$, with $\beta_1 = \beta_2 = \beta_3 = 1$ and hence
\[
\Im z \sim   -\frac {\beta_1 +\beta_3}{2(x_3-x_1)} h \log(1/h)  \approx -3 \cdot 10^{-7}.
\]
for the image on the left and $\beta_1 = .9, \beta_2 = .1, \beta_3 = 1$ and hence
\[
 \Im z_+ \sim   -\frac {\beta_3  -\beta_2}{2(x_3-x_2)} h \log(1/h)  \approx -1.5 \cdot 10^{-6}, 
\]
\[
\Im z_- \sim   -\frac {\beta_1  +\beta_2}{2(x_2-x_1)} h \log(1/h)  \approx -1.4 \cdot 10^{-6},
\]
for the image on the right.  These plots match the results of our theorem perfectly.

In Figure \ref{fig:Ndelta}, we demonstrate that in the case of either $N=5$ or $N=6$, we can achieve a variety of outcomes.  Indeed, setting $\beta_j = 1$ for all $j$, we observe what appears to be a single line of resonances looking at the figures on the left.  Selecting $\beta$ values that lead to different interaction strengths, we convincingly observe three resonance lines in the top right plot computed with $x_1 = -5$, $x_2  = -\sqrt{2}$, $x_3 = 0$, $x_4 = 2 \sqrt{2}$, $x_5 = 7$ and $\beta_1 = 1$, $\beta_2 = .6$, $\beta_3 = .1$, $\beta_4 = .6$, $\beta_5 = 1$.  A similar result holds for $6$ $\delta$ functions in the bottom right plot using $x_1 = -7$, $x_2  = -2\sqrt{2}$, $x_3 = -\pi/4$, $x_4 = \sqrt{2}$, $x_5 = e$, $x_6 = 5$ and $\beta_1 = 1$, $\beta_2 = .1$, $\beta_3 = .5$, $\beta_4 = .2$, $\beta_5 = .5$, $\beta_6 = 1$.  

We thus observe that while it does appear possible to generate
multiple strings of resonances, the $2^{N-1}-1$ upper bound of
Theorem \ref{t:Ndelta} may be far from optimal.
 Indeed, it is intriguing
that in the case when all $\beta$ values are equal, the numerical
findings are that there is only one string of resonances.  Hence there may,
for instance, be
symmetry reductions that allow us to dramatically improve
the bounds on the number of resonance lines.

\bibliography{all}{}
\bibliographystyle{plain}

\end{document}